\documentclass[a4paper]{article}
\usepackage[utf8]{inputenc}
\usepackage{authblk}
\usepackage{amsthm}
\usepackage{amsmath}
\usepackage{amssymb}
\usepackage{amsfonts}
\usepackage[normalem]{ulem}
\usepackage[colorlinks]{hyperref}

\newtheorem{theorem}{Theorem}
\newtheorem{corollary}[theorem]{Corollary}
\newtheorem{lemma}[theorem]{Lemma}
\newtheorem{proposition}[theorem]{Proposition}
\newtheorem{definition}[theorem]{Definition}
\newtheorem{remark}[theorem]{Remark}

\newtheorem{assumption}[theorem]{Assumption}

\newcommand{\cG}{\mathcal{G}}
\newcommand{\R}{\mathbb{R}}
\newcommand{\N}{\mathbb{N}}
\newcommand{\Z}{\mathbb{Z}}

\title{Existence of stationary solutions of supercritical nonlinear Schrödinger equations on some metric graphs}

\author{Jack Borthwick and Nabile Boussaïd}

\affil{\small{McGill University, Department of Mathematics and Statistics, Burnside Hall, 805 Sherbrooke West, Montréal, Québec, H3A 0B9, Canada \\
Universit\'{e} Marie et Louis Pasteur,
CNRS, Institut UTINAM,
\'{E}quipe de physique th\'{e}orique,
F-25000 Besan\c{c}on, France}
}

\date{20/11/2024}

\begin{document}

\maketitle

\begin{abstract}
We consider the existence of stationary wave solutions with prescribed mass to a
supercritical nonlinear Schr\"odinger equation on a noncompact connected metric graph without a small mass assumption.
\end{abstract}

\section{Introduction}
In this work, we consider the
existence of stationary wave solutions with prescribed mass to a
supercritical nonlinear Schr\"odinger equation on a \emph{noncompact connected metric graph} $\cG$:
\begin{equation}\label{Eq:NLS}\tag{NLS}
-u''+\lambda u=|u|^{p-2}u,
\end{equation}
with mass supercritical power, that is
$p>6$, and
Kirchhoff
(namely current conservation)
conditions
\begin{equation}\label{Eq:Kirchoff}\tag{KC}
\sum_{{\rm e} \succ {\rm v}}u_{\upharpoonright\rm e}'({\rm v})=0,\, \forall{\rm v} \in \mathcal{V}.
\end{equation}

The first work to settle a very similar question is due to Chang, Jeanjean and Soave~\cite{MR4755505}, in the case where the graph is compact. Since the graph is compact, the energy functional admits a minimum under a mass constraint and the authors considered minimizing sequences. The question of extending such a result to non-compact graphs was then considered by the first author with Chang, Jeanjean and Soave~\cite{BorthwickChangJeanjeanSoaveApp} in the case where the nonlinearity is localized to the compact core of a graph with finitely many edges. There infimum of the energy is not finite. Thus the authors considered the existence of critical points for the mass constrained energy through the analysis of Palais-Smale sequences. The existence of such sequences with the some informarmation on the Morse index asked for a new abstract theory on bounded Palais-Smale sequence considered by these authors~\cite{BCJS-2022}. The restriction to a localized nonlinearity and the finite number of edges was then released by Dovetta, Jeanjean and Serra~\cite{dovetta2024normalized} up to a smallness assumption on the mass together with some geometrical assumptions on the graph. Interestingly, they considered periodic graphs. The purpose of the present analysis is to release as much as possible the small mass constraint and the the geometrical assumption. We make an assumptions on the bottom of the spectrum and the absence of positive solutions with zero frequency. We provide examples of graphs fulfilling these assumptions.

\section{Framework}

Recall a metric graph is a couple $\cG = (\mathcal{E}, \mathcal{V})$, where $\mathcal{E}$ is a
set of real intervals called edges\footnote{Each interval is from a different copy of $\R$. The same interval can appear several times.}
and $\mathcal{V}$ is the set of the endpoints, called vertices\footnote{More precisely this is a partition of the set of endpoints of
$\mathcal{E}$ again from different copies of $\R$. Elements from the same set in the partition are identified to model the connection between edges. An edge is closed if its ends are identified. Two vertices can be connected by several edges.},
of the edges
and ${\rm e} \succ {\rm v}$ denotes all the edges ${\rm e}\in {\mathcal E}$ such that ${\rm v}$ is an endpoint of ${\rm e}$ while $u_{\upharpoonright\rm e}$ is the restriction of $u$ to ${\rm e}$.
Any bounded edge ${\rm e}$ is identified with a closed bounded interval\footnote{There is no orientation and each edge is determined up to an isometry and hence only its length is relevant.},
e.g. $I_{\rm e}=[0,|{\rm e}|]$ (where
$|{\rm e}|$ is the length of ${\rm e}$), while each unbounded edge is identified with a closed half-line, e.g. $I_{\rm e}=[0,+\infty)$.
The notation $u_{\upharpoonright\rm e}'({\rm v})$ stands for the derivative outward to the vertex into the edge, e.g. $u_{\upharpoonright\rm e}'(0)$,
if the vertex ${\rm v}$ is identified with $0$,
$-u_{\upharpoonright\rm e}'(|{\rm e}|)$, if it is identified with
$\pi_{\mathrm{e}}$.

We prove the existence of solutions
to~\eqref{Eq:NLS}
with prescribed norm (or mass)
as critical points of the energy functional
\begin{equation}\label{Eq:Energy}\tag{E}
E(u,\cG ):=\frac{1}{2}\int_{\cG }|u'|^2\, \,\mathrm{d}x-\frac{1}{p}\int_{\cG }|u|^p\, \,\mathrm{d}x
\end{equation}
under the constraint
\begin{equation}\label{Eq:Constraint}
\int_{\cG }|u|^2\,\,\mathrm{d}x=\mu>0.
\end{equation}
In other words, we prove the existence of critical points of the functional $E(u,\cG )$ constrained on the $L^2$-sphere, e.g. in
\[
H_{\mu}^1(\cG ):= \left\{u\in H^1(\cG ): \int_{\cG }|u|^2\,\,\mathrm{d}x=\mu\right\}.
\]
If $u\in H^1_{\mu}(\cG )$ is such a critical point, then there exists a Lagrange multiplier $\lambda\in \mathbb{R}$ such that $u$ satisfies the following problem
\begin{equation}\label{Eq:Kirchhoff}
\begin{aligned}
\forall {\rm e} \in \mathcal{E},\quad
&
-u_{\upharpoonright\rm e}''+\lambda u_{\upharpoonright\rm e}=|u_{\upharpoonright\rm e}|^{p-2}u_{\upharpoonright\rm e}
\\
\forall {\rm v} \in \mathcal{V},\quad
&
\sum_{{\rm e} \succ {\rm v}}u_{\upharpoonright\rm e}'({\rm v})=0.
\end{aligned}
\end{equation}
For our proof to hold, we make the two following assumptions.

\begin{assumption}\label{Assump:EssSpec}
The spectrum of $-\Delta_\cG$, the Kirchoff lapalacian, is such that
\[
 \inf\sigma(-\Delta_\cG)=0.
\]
\end{assumption}
See Section~\ref{Sec:EssSpec} for examples.

\begin{assumption}\label{Assump:NoZeroPostiveZerostate}
If $u\in H^1(\cG )$
satisfies~\eqref{Eq:Kirchhoff} with $\lambda=0$ and $u\geq0$
then $u\equiv 0$.
\end{assumption}
See Section~\ref{Sec:NoZeroPostiveZerostate} for examples.
\begin{remark}
The Gagliardo-Nirenberg inequality, see~\eqref{Eq:GN} below, together with
the ideas from the Appendix,
impose some constraint on the $L^\infty$-norm of solutions to~\eqref{Eq:Kirchhoff} with respect to the mass $\mu$. From the ideas of the Appendix, on the other hand, the maximum of a positive solution is bounded by means of the supremum of the lengths of the edges. Therefore, if $\mu$ is small no positive solution to~\eqref{Eq:Kirchhoff} is available. Together with Remark~\ref{Rem:EssSpec}, below, this shows that Theorem~\ref{Thm:Main}, below, is valid if $\mu$ is small enough without imposing Assumptions~\ref{Assump:EssSpec} nor~\ref{Assump:NoZeroPostiveZerostate}.
\end{remark}

Here is the main result of this analysis.
\begin{theorem}\label{Thm:Main}
Let
$\cG$ be
noncompact connected metric graph
such that Assumptions~\ref{Assump:EssSpec} and~\ref{Assump:NoZeroPostiveZerostate} hold then for any fixed $\mu>0$,
there exists $(u, \lambda)\in H^1_\mu(\cG)\times \mathbb{R}^+$, $u>0$, which solves
\eqref{Eq:Kirchhoff}.
\end{theorem}

\paragraph{Notations.}

A metric graph $\cG$ is naturally endowed with a metric space structure corresponding to the infimum of the path lengths between two points. The corresponding space of complex valued continuous functions is denoted $C(\cG)$ and $C_c(\cG)$ is the subspace of continuous functions with compact support.

The space $L^p(\cG)$ is the set
\[
 \left\{(u_{\upharpoonright\rm e})_{{\rm e}\in {\mathcal E}}\in \otimes_{{\rm e}\in {\mathcal E}} L^p({\rm e}),\, \sum_{{\rm e}\in {\mathcal E}} \|u_{\upharpoonright\rm e}\|_p^p<+\infty\right\}
\]
endowed with the norm
\[
 \|u\|_{L^p(\cG)}= \left(\sum_{\mathrm{e}\in {\mathcal E}} \|u_{\upharpoonright\rm e}\|_{L^p(\mathrm{e})}^p\right)^{1/p}.
\]
The Sobolev space $H^1(\cG)$ is the
completions of $C_c(\cG)$ with respect to the norm
\[
\|u\|_{H^1(\cG)} = \left(\sum_{{\rm e}\in {\mathcal E}} \left(\|u_{\upharpoonright\rm e}'\|_{L^2(\mathrm{e})}^2 + \|u_{\upharpoonright\rm e}\|_{L^2(\mathrm{e})}^2\right)\right)^{1/2}
\]
or equivalently the subspace of $L^2(\cG)$:
\[
 \left\{u\in C(\cG)\cap L^2(\cG),\, (u_{\upharpoonright\rm e}')_{\rm e}\in L^2(\cG)\right\}.
\]

\section{Properties of quantum graphs and examples}

A quantum graph is a metric graph endowed with a self-adjoint operator. A usual reference for quantum
graphs is the monograph by G. Berkolaiko and P. Kuchment~\cite{BerkolaikoKuchment}.

The quadratic form
\[
 u\to \int_\cG |u'|^2
\]
is non negative and closed, see~\cite[Section VIII.8]{ReedSimon},
on $H^1_0(\cG)$ and
$H^1(\cG)$. This corresponds
respectively
to
the Dirichlet and Neumann
laplacian
with Kirchoff boundary conditions at the
inner
vertices.

The degree, $d({\rm v})$, of a vertex ${\rm v}$ is defined as the number of edges ${\rm v}$ such that ${\rm e}\succ {\rm v}$. Vertices of degree $1$ are endpoints of the graph. If there are no such vertex then Dirichlet and Neumann laplacians coincide. Note that vertices of degree $2$ are in a sense inessential since on can replace the two edges by one having the same total length. There is one exception if the graph has one vertex and thus one edge, this then correspond to the torus of dimension $1$. This could be overcome if the degree count only distinct edges then in the torus example, the one vertex is of degree one and not an end. We prefer to avoid this and work with the torus exception.

Although,
in this analysis,
we can consider the Dirichlet laplacian,
we will focus on the Neumann one since for vertices of degree $1$ they coincide with the Kirchoff condition. We now denote it $-\Delta_\cG$ and call it Kirchoff laplacian.

\subsection{Spectral properties of the Kirchoff laplacian}\label{Sec:EssSpec}

The main goal of this subsection is to discuss examples for which Assumption~\ref{Assump:EssSpec}
holds.

For an account of the knowledge on the spectrum on quantum graphs, we refer to the monograph by Kurasov~\cite{KostenkoNicolussi}. In particular, see~\cite[Remark 2.7]{KostenkoNicolussi},
$0$ is an eigenvalue if and only if non zero constant functions are integrable. Moreover, if the
graph is unbounded, in the sense that
\[
 \sum_{{\rm e}\in {\mathcal E}}|{\rm e}|=+\infty,
\]
the bottom of the spectrum is zero if the bottom of the essential spectrum is zero.

A condition ensuring that the bottom of the spectrum is zero, is, see~\cite[Corollary 3.18.(iii)]{KostenkoNicolussiBook}, that intrinsic size is infinite, that is
\[
 \sup_{{\rm e}\in {\mathcal E}}|{\rm e}|=+\infty.
\]
Indeed, using Rayleigh quotient,
\[
 \inf\sigma(-\Delta_\cG)=\inf_{f\in H^1(\cG), f\neq 0}\frac{\|f'\|_2}{\|f\|_2},
\]
this follows from consider the usual spreading tent functions on sequences of edges of lengths tending to infinity.

But this not a sufficient condition, we can refer to~\cite[Appendix]{BoniDovettaSerra}, where it is shown that periodic graphs, see~\cite[Definition 4.1.1.]{BerkolaikoKuchment}, the infimum of the Rayleigh quotient is $0$.

Examples of graph such that $\inf\sigma(-\Delta_\cG)>0$ are provided by equilateral Bethe lattices
(equilateral tree with vertices of constant degree larger than $3$), see~\cite[Example 8.3]{KostenkoNicolussi}, or exponentially growing antitrees, see~\cite[Example 8.6]{KostenkoNicolussi}.
We refer to~\cite[Appendix]{KostenkoNicolussi} and~\cite[Chapter 8]{KostenkoNicolussiBook} for a more detailed discussion where other examples are presented.

\subsection{Concave functions}\label{Sec:NoZeroPostiveZerostate}

In this subsection, we discuss the validity of Assumption~\ref{Assump:NoZeroPostiveZerostate}.
Notice that if $u\in H^2(\cG)$ and $u>0$ then
\[
u_{\upharpoonright\rm e}''=-u_{\upharpoonright\rm e}^{p-1}
\]
and $u$ is a concave function on each edge of the graph.

Let us recall that on $\R$ a regular concave function is always below its tangents. Unless it is constant, some of the tangents vanish and hence there are no regular positive concave function on the real line. On a half-line such as $(0,+\infty)$, $x\mapsto\ln(1+x)$ provides an example of regular positive concave function.
While a non constant regular positive concave function with a critical point
on an open half line vanishes somewhere. In the context of metric graphs show that of there is an infinite edge then there is no positive concave function which tends to $0$ at its infinite end.

On a rooted regular tree\footnote{A tree is a connected and simply connected graph, that is a connected graph with no closed path. A path is a connected sequence of edges where two consecutive edges are distinct. A root on a tree is a distinguished degree one vertex.}
of degree $2$, such that the edges of the $n$-th level from the root
are of length $2^{n-1}$, there is positive function affine on each edge of the $n$-th level to $2^{-n+1}-2^{-n}x$,
each is identified with $[0,2^{n-1}]$. So we can exclude positive continuous piecewise concave functions tending to $0$ and satisfying the Kirchoff condition. Note that here the function is also piecewise convex.

Let $\cG$ be a $\Z$-periodic graph with fondamental domain $W_0$, see~\cite[Definition 4.1.1.]{BerkolaikoKuchment}.
Let us denote by $\tau$ the corresponding group action of $\Z$ on $\cG$. Let $W_k=\tau(k)W_0$, $k\in\Z$. As a set, $\cG=\cup_{k\in\Z} W_k$. An edge is called entering in $W_0$, if either it is in $W_{-1}$ with a final vertex in $W_0$
or it is in $W_0$ with a starting vertex in $W_{-1}$. Denote the corresponding set by ${\cal E}^+_0$. Similarly, define entering edges to $W_k$: ${\cal E}^+_k=\tau(k){\cal E}^+_0$. Define exiting edges ${\cal E}^-_0$ and then ${\cal E}^-_k$. Note that ${\cal E}^+_k={\cal E}^-_{k-1}$ and all these sets are finite and their cardinals are uniformly bounded.

Now let $\widetilde{W}_k={\cal E}^+_k\cup W_k \cup {\cal E}^-_k$ as a subgraph of $\cG$, in the sense that all the common vertices are identified. Hence $\widetilde{W}_k$ includes $W_k$ and all the edges connected to it as well as the vertices.
Using Kirchoff condition, we obtain
\[
 \sum_{{\rm v}\in W_k } \sum_{{\rm e}\succ {\rm v},{\rm e}\in \widetilde{W}_k}u_{\upharpoonright{\rm e}}'({\rm v})=0
\]
or
\[
 \sum_{{\rm e}\in \widetilde{W}_k}
 \sum_{{\rm e}\succ {\rm v},{\rm v}\in W_k }
 u_{\upharpoonright{\rm e}}'({\rm v})=0.
\]
Note that if $u''< 0$ then for any ${\rm v}\in W_k$
\[
\sum_{{\rm e}\in {\cal E}, {\rm e}\succ {\rm v}}u_{\upharpoonright{\rm e}}'(v)< 0
\]
and hence
\[
 \sum_{{\rm v}\in \widetilde{W}_k\setminus W_k } \sum_{{\rm e}\succ {\rm v},{\rm e}\in \widetilde{W}_k}u_{\upharpoonright{\rm e}}'({\rm v})>0
\]
or
\[
 \sum_{({\rm e}, {\rm v})\in {\cal E}^-_k,\, {\rm e}\succ {\rm v}}u_{\upharpoonright{\rm e}}'({\rm v})
 <
  \sum_{({\rm e}, {\rm v})\in {\cal E}^+_k,\, {\rm e}\succ {\rm v}}u_{\upharpoonright{\rm e}}'({\rm v}).
\]
Let
\[
  \sigma_k:=
   \sum_{({\rm e}, {\rm v})\in {\cal E}^-_k,\, {\rm e}\succ {\rm v}}u_{\upharpoonright{\rm e}}'({\rm v}),
\]
using again $u''<0$, we eventually get $\sigma_{k-1}<\sigma_k$. If $u\in H^2(\cG)$ then $u'\to 0$ at infinity and so does the sequence $(\sigma_k)_{k\in \Z}$. This is a contradiction. Therefore $\Z$-periodic graphs satisfy Assumption~\ref{Assump:NoZeroPostiveZerostate}.

\paragraph{Aknowledgements}

This work has been supported by the EIPHI Graduate School (contract ANR-17-EURE-0002) and  by the Bourgogne-Franche-Comté Region and funded in whole or in part by the French National Research Agency (ANR) as part of the QuBiCCS project "ANR-24-CE40-3008-01``.

\section{Some preliminary definitions and results}

\emph{In the rest of this analysis, we consider $\cG$ to be
noncompact connected metric graph such that Assumptions~\ref{Assump:EssSpec} and~\ref{Assump:NoZeroPostiveZerostate} hold.}

Our strategy closely follows~\cite{BorthwickChangJeanjeanSoaveApp} with some important modifications to include a distributed nonlinear potential.

We first consider an auxiliary family of energy functionals
$E_{\rho}(\cdot,\cG ): H^1(\cG ) \to \mathbb{R}$,
$\rho$ in $\left[\frac12, 1 \right]$,
given by
\begin{align}\label{Eq:Energyrho}
 E_{\rho}(u,\cG ):=\frac{1}{2}\int_{\cG }|u'|^2\,\mathrm{d}x
 -\frac{\rho }{p}\int_{\cG }|u|^p\,\mathrm{d}x, \qquad \forall u\in H^1(\cG ).
\end{align}
Since we are focusing in $\rho$ close to $1$. The condition $\rho\geq 1/2$ is arbitrary. It is meant to emphasize that we do not consider
$\rho$ vanishing.

We make a crucial use of the abstract analysis~\cite{BCJS-2022} and more precisely~\cite[Theorem 1]{BCJS-2022}. Let us recall this result. Beforehand, let us mention that the existence result we are aiming to, boils down to the
convergence of Palais-Smale sequences induced by a mountain pass structure on the functional
$E_\rho(u,\cG)$ from~\eqref{Eq:Energyrho} constrained to $H_{\mu}^1(\cG )$. The lack of compactness is compensated by a bound on the corresponding sequence of Morse indices.

Let $\left(E,\langle \cdot, \cdot \rangle\right)$ and $\left(H,(\cdot,\cdot)\right)$ be two \emph{infinite-dimensional} Hilbert spaces and assume that $E\hookrightarrow H \hookrightarrow E'$,
with continuous injections.  For simplicity, we assume that the continuous injection $E\hookrightarrow H$ has norm at most $1$ and identify $E$ with its image in $H$.
Set
\[\|u\|^2:=\langle u,u \rangle,\quad |u|^2:=(u,u), \, u\in E,\]
and, for $\mu>0$,
\[
S_\mu:= \{ u \in E\big|\, |u|^2=\mu \}.
\]
In the context of this analysis, we shall have $E=H^1(\cG)$ and $H=L^2(\cG)$.
Clearly, $S_{\mu}$ is a smooth submanifold of $E$ of codimension $1$. Its tangent space at a given point $u \in S_{\mu}$ can be considered as the closed codimension $1$ subspace of $E$ given by:
\[  T_u S_{\mu}= \{v \in E \big|\, (u,v) =0 \}.\]
In the following definition, we denote by $\|\cdot\|_*$ and $\|\cdot\|_{**}$, respectively, the operator norm of $\mathcal{L}(E,\R)$ and of $\mathcal{L}(E,\mathcal{L}(E,\R))$.

\begin{definition}
Let $\phi : E \rightarrow \mathbb{R}$ be a $C^2$-functional on $E$ and $\alpha \in (0,1]$. We say that $\phi'$ and $\phi''$ are $\alpha$-H\"older continuous on bounded sets if for any $R>0$ one can find $M=M(R)>0$ such that, for any $u_1,u_2\in B(0,R)$:
\begin{equation}\label{Holder}
||\phi'(u_1)-\phi'(u_2)||_*\leq M ||u_1-u_2||^{\alpha}, \quad ||\phi''(u_1)-\phi''(u_2)||_{**} \leq M||u_1-u_2||^\alpha.
\end{equation}
\end{definition}

\begin{definition}\label{def D}
Let $\phi$ be a $C^2$-functional on $E$. For any $u\in E$, we define the continuous bilinear map:
\[ D^2\phi(u)=\phi''(u) -\frac{\phi'(u)\cdot u}{|u|^2}(\cdot,\cdot). \]
\end{definition}

\begin{definition}\label{def: app morse}
Let $\phi$ be a $C^2$-functional on $E$. For any $u\in S_{\mu}$ and $\theta >0$, we define
 the \emph{approximate Morse index} $\tilde m_\theta(u)$ by
\[
\sup \left\{\dim\,L\mid \, L \text{ subspace of $T_u S_{\mu}$ s.t.: }
D^2\phi(u)[\varphi, \varphi]\leq-\theta \|\varphi\|^2, \, \forall \varphi \in L \right\}.
\]
If $u$ is a critical point for the constrained functional $\phi|_{S_{\mu}}$ and $\theta=0$, we say that this is the \emph{Morse index of $u$ as constrained critical point}.
\end{definition}

The following abstract theorem was established in \cite{BCJS-2022}. Its derivation is based on a combination of ideas from \cite{FG-1992,FG-1994,J-PRSE1999} implemented in a convenient geometric setting.  Related theorems, but dealing with unconstrained functionals, are developed in \cite{BeRu,LoMaRu}.

\begin{theorem}[Theorem 1 in \cite{BCJS-2022}]\label{Thm:mp_geom}
Let $I \subset (0,+\infty)$ be an interval and consider a family of $C^2$ functionals $\Phi_\rho: E \to \mathbb{R}$ of the form
\[
\Phi_\rho(u) = A(u) -\rho B(u), \qquad \rho \in I,
\]
where $B(u) \ge 0$ for every $u \in E$, and
\begin{equation}\label{hp coer}
\text{either $A(u) \to +\infty$~ or $B(u) \to +\infty$ ~ as $u \in E$ and $\|u\| \to +\infty$.}
\end{equation}
Suppose moreover that $\Phi_\rho'$ and $\Phi_\rho''$ are $\alpha$-H\"older continuous on bounded sets for some $\alpha \in (0,1]$.
Finally, suppose that there exist $w_1, w_2 \in S_{\mu}$ (independent of $\rho$) such that, setting
\[
\Gamma= \left\{ \gamma \in C([0,1],S_{\mu})\big| \ \gamma(0) = w_1, \quad \gamma(1) = w_2\right\},
\]
we have
\begin{equation*}
c_\rho:= \inf_{\gamma \in \Gamma} \ \max_{ t \in [0,1]} \Phi_\rho(\gamma(t)) > \max\{\Phi_\rho(w_1), \Phi_\rho(w_2)\}, \quad \rho \in I.
\end{equation*}
Then, for almost every $\rho \in I$, there exist sequences $\{u_n\} \subset S_{\mu}$ and $\zeta_n \to 0^+$ such that, as $n \to + \infty$,
\begin{enumerate}
\item\label{Thm:mp_geom_1} $\Phi_\rho(u_n) \to c_\rho$;
\item\label{Thm:mp_geom_2} $||\Phi'_\rho|_{S_{\mu}}(u_n)||_* \to 0$;
\item\label{Thm:mp_geom_3} $\{u_n\}$ is bounded in $E$;
\item\label{Thm:mp_geom_4} $\tilde m_{\zeta_n}(u_n) \le 1$.
\end{enumerate}
\end{theorem}

\begin{remark}\label{Rem:WeakCV}
An immediate observation, see \cite[Remarks 1.3]{BCJS-2022}, following from Theorem~\ref{Thm:mp_geom}.(\ref{Thm:mp_geom_2})-(\ref{Thm:mp_geom_3}) is that
\begin{equation}\label{free-gradient}
 \Phi'_{\rho}(u_n) + \displaystyle \lambda_n (u_n, \cdot) \to 0 \, \mbox{ in }  E' \mbox{ as } n \to + \infty
\end{equation} where we have set
\begin{equation}\label{def-almost-Lagrange}
\lambda_n := - \displaystyle \frac{1}{\mu}( \Phi'_\rho(u_n)\cdot u_n).
\end{equation}
The sequence $\{\lambda_n\} \subset \R$ defined in \eqref{def-almost-Lagrange}
is called
sequence of \emph{almost Lagrange multipliers}.

\end{remark}

\begin{remark}\label{Rem:LouisProperty}
 Theorem~\ref{Thm:mp_geom}.(\ref{Thm:mp_geom_4}) directly implies that if there exists a subspace $W_n \subset T_{u_n}S_{\mu}$ such that
\begin{equation}\label{*lem-Louis1}
D^2\Phi_{\rho}(u_n)[w,w] = \Phi''_{\rho}(u_n)[w,w] + \displaystyle \lambda_n (w,w) < - \zeta_n ||w||^2,\, \forall  w \in W_n \setminus \{0\},
\end{equation}
then necessarily $\dim W_n \leq 1$.
\end{remark}

\section{Mountain pass solutions}\label{sec: appr}

To show that the family  $E_\rho(\cdot\,\cG)$ enters into the framework of Theorem~\ref{Thm:mp_geom} we first need to show that it has a mountain pass geometry on $H^1_\mu(\cG)$ uniformly with respect to $\rho\in \left[\frac12, 1 \right]$.

\paragraph{Gagliardo-Nirenberg inequalities on graphs.}
For every
graph $\cG$, recall Gagliardo-Nirenberg inequalities :
for any $p\geq2$, there exists $K_{p,\cG}>0$ (depending on $p$ and $\cG$ only) such that
for all $u$ in $H^1(\cG)$
\begin{equation}\label{Eq:GN}
\|u\|_{p}^p\leq K_{p,\cG}\|u\|_{2}^{\frac p2+1}\|u'\|_{2}^{\frac p2-1},\quad
\|u\|_{\infty}\leq \sqrt{2}\|u\|_{2}^{\frac12}\|u'\|_{2}^{\frac12}.
\end{equation}
We refer for instance \cite[Section 2]{AST_JFA2016}.  Note that the second statement follows from the fundamental theorem of calculus (e.g.~\cite[Lemma 1.2.8]{BK}) while the former
is an interpolation of the $L^2(\cG)$-norm with the second. We also refer to~\cite{Esteban20} for further comments and bibliographical aspects.

\paragraph{Mountain pass geometry.}
\begin{lemma}\label{MP-geometry}
For every $\mu>0$,
there exist $w_1, w_2 \in S_\mu$ independent of $\rho \in \left[\frac12,1\right]$ such that
\begin{align*}
c_{\rho}:=\inf\limits_{\gamma\in
\Gamma} \, \max\limits_{t\in[0,1]}E_{\rho}(\gamma(t),\cG ) > \max\{E_{\rho}(w_1,\cG ), E_{\rho}(w_2,\cG )\}, \qquad \forall
\rho\in \left[\frac{1}{2},1\right],
\end{align*}
where
$$\Gamma:=\left\{\gamma\in C([0,1], H^1_{\mu}(\cG ))\big| ~\gamma~ \mbox{is continuous},~\gamma(0)=w_1,\gamma(1)=w_2\right\}.$$
Moreover, there exists $\kappa>0$ such that
\[
 \forall \rho\in \left[\frac{1}{2},1\right], c_\rho\geq \kappa.
\]
\end{lemma}
\begin{proof}
The proof is borrowed
from~\cite[Lemma 3.1]{BorthwickChangJeanjeanSoaveApp}.
We provide it for the reader's convenience with some due
adaptations since the framework is different.

For any $\mu, k>0$, denote
\begin{equation*}
A_{\mu,k}:=\{u\in H_{\mu}^1(\cG )\big| \int_{\cG }|u'|^2\,\mathrm{d}x<k\}.
\end{equation*}
First note that, due to Assumption~\ref{Assump:EssSpec},
$A_{\mu,k} \neq \emptyset$
whenever $k>0$
and, since $H_{\mu}^1(\cG )$ is connected,
\[
\partial A_{\mu,k}=\{u\in H_{\mu}^1(\cG )\big| \int_{\cG }|u'|^2\,\mathrm{d}x=k\}
\neq \emptyset.
\]

Now, by the Gagliardo-Nirenberg inequality~\eqref{Eq:GN}, we obtain
for all $u$ in $\partial A_{\mu, k}$
\begin{align*}
E_{\rho}(u,\cG )
&\ge\frac{1}{2}\|u'\|_{L^2(\cG )}^2
- \rho\frac{K_{p,\cG}}{p}\mu^{\frac{p}{4}+\frac{1}{2}}\|u'\|_{L^2(\cG )}^{\frac{p}{2}-1}\\
&\ge \frac{1}{2}k\left(1
- \left(\frac{k}{k_1}\right)^\frac{p-6}{4}\right)
\end{align*}
where
$$k_1:=\left(\frac{p}{2\rho K_{p,\cG}}\right)^\frac{4}{p-6}\mu^{-\frac{p+2}{p-6}}.$$
Let
$$\kappa_\rho(k):=\frac{1}{2}k\left(1
- \left(\frac{k}{k_1}\right)^\frac{p-6}{4}\right)
$$
then for $k_0\in(0, k_1)$, for any $\mu>0$ and $u\in \partial A_{\mu, k_0}$
we have
\begin{equation*}
\inf_{u\in \partial A_{\mu,k_0}}E_{\rho}(u,\cG )
\ge \kappa_\rho(k)>0.
\end{equation*}

Next, observe that for any $u \in A_{\mu,k}$
\begin{align*}
E_{\rho}(u,\cG )\le\frac{1}{2}\int_{\cG }|u'|^2\,\,\mathrm{d}x
\le\frac{1}{2}k.
\end{align*}
Since $A_{k,\mu} \neq \emptyset$ for all $k>0$,
it is possible to choose a $w_1 \in H^1_\mu(\cG)$ such that
\begin{equation}\label{mpw1}
\|w_1'\|_{L^2(\cG)}^2 <k_0 \quad \text{and} \quad E_\rho(w_1, \cG) \le \kappa_\rho(k_0),\,
\forall \rho \in \left[\frac12,1\right].
\end{equation}
Moreover, we also observe that we can identify any edge,
say ${\mathrm e_1}$, with the interval $[-\pi_{1}/2, \pi_{1}/2]$ if it is bounded
or $[-\pi_{1}/2, +\infty)$ if not. It follows that any compactly supported $H^1$
function $w$ on $[-\pi_{1}/2, \pi_{1}/2]$, with mass $\mu$,
can be seen as a function in $H^1_\mu(\cG)$. Defining $w_t(x) := t^{1/2} w(t x)$,
with $t >1$, we have $w_t \in H^1_\mu(\cG)$
and that
\begin{align*}
E_\rho(w_t, \cG) &= \frac{t^2}{2} \int_{{\mathrm e_1}} |w'|^2\,\,\mathrm{d}x- \frac{\rho t^{\frac{p-2}{2}}}{p}\int_{{\mathrm e_1}} |w|^p\,\,\mathrm{d}x\\
&\le  \frac{t^2}2\left( \int_{{\mathrm e_1}} |w'|^2\,\,\mathrm{d}x - \frac{2\rho t^{\frac{p-6}{2}}}{p}\int_{{\mathrm e_1}} |w|^p\,\,\mathrm{d}x\right)\\
\end{align*}
for every $\rho \in [\frac{1}{2}, 1]$.
Since $p>6$, then the right-hand side tends to $-\infty$ as $t \to +\infty$, and in particular there exists $t_2>0$ large enough such that $\|w'_t\|^2_{L^2(\cG )}=t^2 \|w'\|_{L^2(\cG)}^2>2k_0$ and $E_{\rho}(w_{t},\cG )<0$ for all $t> t_2$ and $\rho\in[\frac{1}{2}, 1]$. We set $w_2:=w_{2t_2}$, and we point out that
\begin{equation}\label{mpw2}
\|w_2'\|_{L^2(\cG)}^2 >2k_0 \quad \text{and} \quad E_\rho(w_2, \cG) <0\quad \forall \rho \in \left[\frac12,1\right].
\end{equation}
At this point the thesis follows easily. Let $\Gamma$ and $c_\rho$ be defined as in the statement of the lemma for our choice of $w_1$ and $w_2$; the fact that $\Gamma \neq 0$ is straightforward, since
\[
\gamma_0(t):=\frac{\mu^{1/2}}{\| (1-t) w_1 + t w_2\|_{L^2(\cG)}}  \left[(1-t) w_1 + t w_2\right] \qquad t \in [0,1]
\]
belongs to $\Gamma$. By \eqref{mpw1} and \eqref{mpw2}, we note that for every $\gamma \in \Gamma$ there exists $t_\gamma \in [0,1]$ such that $\gamma(t_\gamma) \in \partial A_{\mu, k_0}$, by continuity. Therefore, for any $\rho \in [1/2,1]$, we have for every $\gamma \in \Gamma$
\[
\max_{t \in [0,1]} E_\rho(\gamma(t),\cG) \ge E_\rho(\gamma(t_\gamma),\cG) \ge \inf_{u \in \partial A_{\mu, k_0}} E_\rho(u, \cG) \geq  \kappa_\rho(k_0)
\]
and thus $c_\rho \ge \kappa_\rho(k_0)$  while
\[
\max\{E_\rho(w_1,\cG), E_\rho(w_2,\cG)\} = E_\rho(w_1,\cG) < \frac{\alpha}2.
\]
To conclude we remark that
$\kappa_\rho(k_0)\geq \kappa_1(k_0)$ for $\rho \in [1/2,1]$. We set $\kappa:=\kappa_1(k_0)$.
\end{proof}

\begin{remark}\label{Rem:EssSpec}
If Assumption~\ref{Assump:EssSpec} is not satisfied then Lemma~\ref{MP-geometry} still holds if $\mu>0$ is small enough for $E(u,\cG)-\sigma \|u\|_2^2$ where
$\sigma:=\inf\sigma(-\Delta_\cG)$.
\end{remark}

\paragraph{Properties of the essential spectrum.}
\begin{lemma}\label{L-eigenvalue}
For any $u\in H^1(\cG)$,
for any $\lambda <0$,
there exists an infinite dimensional subspace $L$ of $H^1(\cG )$
such that
\begin{equation*}
\int_{\cG} \left[ |w'|^2 + \left( \lambda - (p-1)\rho |u|^{p-2}\right) w^2\right]  \, \,\mathrm{d}x
<0, \forall \, w \in L.
\end{equation*}
\end{lemma}
\begin{proof}
Since $\inf\sigma_{\rm ess}(-\Delta_\cG)=0$ and $ V_u : w\mapsto (p-1)\rho|u|^{p-2}w$ is compact from
$H^2(\cG)$ to $L^2(\cG)$
Weyl's criterion gives $0=\inf\sigma_{\rm ess}(-\Delta_\cG-V_u)$ and so
for any $\lambda<0$ there exists an infinite dimensional subspace $L$ of $H^1(\cG)$ such that
\[
 \forall w\in L,\, \int_{\cG} \left[ |w'|^2 + \left( \lambda - |u|^{p-2}\right) w^2\right]  \, \,\mathrm{d}x
<0
\]
which concludes the proof.
\end{proof}

\paragraph{An abstract compactness lemma.}
\begin{lemma}\label{Lem:Abstract}
Let
$\{u_{n}\} \subset H_\mu^1(\cG)$
be a bounded nonnegative sequence.
Let $(\lambda_n)_{n\in\N}\in \R$ and $(\rho_n)_{n\in\N}\in \R$ be convergent
to $\lambda_\infty$ and $\rho_\infty$ respectively
and such that
\begin{equation}\label{Eq:SequencesEnergy}
E_{\rho_n}(u_n)\to c>0\quad\mbox{ and }\quad E'_{\rho_n}(u_n) +  \lambda_n u_n \to 0 \quad \mbox{in the dual of} \quad H^1(\cG )
\end{equation}
and if, for each integer $n$, the inequality
\begin{align}\label{Eq:NegativeSubspace}
E''_{\rho_n}(u_n)[\varphi, \varphi] + \lambda_n ||\varphi||_{L^2(\cG )}^2 < 0
\end{align}
holds for any $\varphi \in W_n \setminus \{0\}$, $W_n$ a subspace of $T_{u_n} S_\mu$ then the dimension of $W_n$ is at most $1$.
Then $\{u_{n}\}$ has a converging subsequence
in $H^1(\cG)$
to some positive $u_\infty$
such that
\begin{equation}\label{Eq:NLSinfty}
- u_{\infty}'' + \lambda_{\infty} u_{\infty}=\rho_\infty u_{\infty}^{p-1}
\end{equation}
with the Kirchhoff condition \eqref{Eq:Kirchhoff} at the vertices.
Moreover, we have $m_0(u_\infty)\leq 1$.
\end{lemma}
\begin{proof}
Up to the extraction of subsequences, we can assume $(\lambda_n)$ and $(\rho_n)$ to be
convergent to $\lambda_\infty$ and $\rho_\infty$, respectively, and
\begin{align}
u_{n}\rightharpoonup u_{\infty} & \mbox{ in}~~ H^1(\cG ),\label{Eq:WeakH1}\\
u_n\to u_{\infty}& \mbox{ in}~~L_{{\rm{loc}}}^r(\cG ), r>2,\label{Eq:LocLp}\\
u_{n}(x)\to u_{\infty}(x)& \mbox{ for a.e.}~ x\in \cG ,\label{Eq:Point}
\end{align}
which implies that $u_{\infty}\ge0$ and $u_{\infty}$ satisfies
\eqref{Eq:NLSinfty}
(with the Kirchhoff condition \eqref{Eq:Kirchhoff} at the vertices).

Let us show that $u_{\infty} \not \equiv 0$. First let us prove that $\lambda_{\infty}\ge0$. Since the codimension of $T_{u_n} S_\mu$ is $1$, we infer that if the inequality \eqref{Eq:NegativeSubspace} holds for every $\varphi \in V_n \setminus \{0\}$ for a subspace $V_n$ of $H^1(\cG)$, then the dimension of $V_n$ is at most $2$.
If $\lambda_n<0$, Lemma~\ref{L-eigenvalue} provides
for $n \in \mathbb{N}$ large,
the existence of a subspace $V_n$ of $H^1(\cG)$ with $\dim V_n \geq 3$ and $a_n>0$ such that,
\[
E''_{\rho_n}(u_n)[\varphi,\varphi] + \lambda_n |\varphi|^2
\le -a_n \|\varphi\|^2,  \qquad \forall  \, \varphi \in V_n \setminus\{0\}.
\]
Therefore, for $n \in \mathbb{N}$ large,
$\lambda_n \ge 0$. Hence $\lambda_\infty\geq 0$

Now from \eqref{Eq:SequencesEnergy} and the fact that $\lambda_n \to \lambda_\infty$, we deduce that
\[
\int_{\cG} \left(u_n' \varphi'+\lambda_n  u_n \varphi\right) \,\,\mathrm{d}x - \rho_n \int_{\cG } u_n^{p-1} \varphi\,\,\mathrm{d}x =o(1) \|\varphi\|_{H^1(\cG)}.
\]
Moreover, by \eqref{Eq:NLSinfty},
\[
\int_{\cG} \left(u_\infty' \varphi'+\lambda_\infty u_\infty \varphi\right) \,\,\mathrm{d}x - \rho_\infty \int_{\cG } u_\infty^{p-1} \varphi\,\,\mathrm{d}x =0
\]
Therefore, taking the difference,
and $\varphi=u_n-u_\infty$,
we infer that
\begin{equation}\label{Eq:Difference}
\int_{\cG }|(u_n-u_\infty)'|^2\, \,\mathrm{d}x +\lambda_{\infty} \int_{\cG }| u_n-u_\infty|^2 \,\,\mathrm{d}x \to 0
\end{equation}
as $n \to \infty$. If we assume that $u_{\infty}  \equiv 0$ then we deduce that
\begin{equation}\label{Eq:Contradiction}
\int_{\cG } |u_n'|^2 \,\,\mathrm{d}x + \lambda_{\infty} \int_{\cG } |u_n|^2 \,\,\mathrm{d}x   \to 0.
\end{equation}
If $\lambda_{\infty} >0$, this is not possible since $\|u_n\|_{L^2(\cG)}^2=\mu>0$. If $\lambda_{\infty} =0$, then
\eqref{Eq:Contradiction} contradicts
the assumption $c>0$ in \eqref{Eq:SequencesEnergy}.
Hence, $u_\infty \not \equiv 0$.

Appealing to the Kirchhoff condition~\eqref{Eq:Kirchhoff} and the uniqueness
in Cauchy-Lipschitz theorem, we have in fact that $u_\infty >0$ in $\cG$. Indeed, assume by contradiction that there exists $x_0\in \cG $ such that $u_\infty(x_0)=0$. If $x_0$ stays in the interior of some edge, then by $u_\infty\ge0$ on $\cG $ it follows that $u'_\infty(x_0)=0$ leading to $u_\infty\equiv 0$ on the whole edge. If instead $x_0$ is a vertex , then by $u_\infty(x)\ge0$, $u'_\infty(x_0)_{\rm e}\geq 0$ for all ${\rm e}\succ x_0$ and by the Kirchhoff condition we also get $u'_\infty(x_0)=0$. Hence, by uniqueness, $(u_\infty)_{\rm e}\equiv 0$, so $u_\infty\equiv 0$ on any edge containing $x_0$; but then, by repeating this argument from first neighbors to next ones and so
iteratively
(since $\cG$
is connected)
we deduce that $u_\infty \equiv 0$ on $\cG$, which is the desired contradiction.

Note that \eqref{Eq:Difference} implies $u_n\to u_\infty$ in $\dot{H}^1(\cG)$ and even
$H^1(\cG)$ if $\lambda_\infty>0$.
Now we claim that $\lambda_\infty>0$.
If $\lambda_\infty=0$, then
from  \eqref{Eq:NLSinfty} contradicts Assumption~\ref{Assump:NoZeroPostiveZerostate}
since $u_\infty\in H^1(\cG)$ and $u_\infty>0$.
So $\lambda_{\infty} >0$ and $u_n \to u_{\infty}$ strongly in $H^1(\cG )$.

It remains to show that the Morse index $m(u_\infty)$ is at most $1$.
If not, in view of Definition~\ref{def: app morse} we may assume by contradiction that there exists a $W_0 \subset T_{u}S_{\mu}$ with $\dim W_0 =2$ such that
\begin{equation}\label{10-24-1}
D^2 E_{\rho_\infty}(u_\infty)[w,w]<0,\,\forall  w \in W_0 \setminus\{0\}.
\end{equation}
Then, since $W_0$ is of finite dimension, there exists $\beta>0$ such that
$$D^2 E_{\rho_\infty}(u_\infty)[w,w]<-\beta,\,\forall  w \in W_0 \setminus\{0\}~~\mbox{with}~~\|w\|_{H^1(\cG )}=1,$$
using the homogeneity of $D^2 E_{\rho_\infty}(u_\infty)$, we deduce that
\begin{equation*}
D^2 E_{\rho_\infty}(u_\infty)[w,w]< -\beta||w||_{H^1(\cG )}^2,\, \forall  w \in W_0 \setminus\{0\}.
\end{equation*}
Now, from \cite[Corollary 1]{BCJS-2022} or using directly that $E_\rho'$ and $E_\rho''$ are $\alpha$-H\"older continuous on bounded sets for some $\alpha \in (0,1]$,
and Sobolev inequality
it follows that there exists $\delta_1>0$ small enough such that, for any $v\in S_{\mu}$ satisfying $||v-u_\infty|| \leq \delta_1$
and $|\rho-\rho_\infty|<\delta_1$,
\begin{equation}\label{L-conditionD2step1}
D^2E_{\rho}(v)[w,w]<-\frac{\beta}{2}||w||_{H^1(\cG )}^2,\,\forall  w \in W_0 \setminus\{0\}.
\end{equation}
Hence, noting that $||u_n-u_\infty||_{H^1(\cG )}\leq\delta_1$ and $|\rho_n-\rho_\infty|<\delta_1$ for $n\in \mathbb{N}$ large enough, we get
\begin{equation}\label{10-24-2}
D^2E_{\rho}(u_n)[w,w]<-\frac{\beta}{2}||w||_{H^1(\cG )}^2,\,\forall  w \in W_0 \setminus\{0\}
\end{equation}
for any such large $n$. Therefore, since $\dim W_0>1$  this contradicts \eqref{Eq:NegativeSubspace}.
Thus we infer that $m_0(u_\infty) \le 1$.
\end{proof}

\section{Main statements}

\begin{proposition}\label{Prop:NLSrho}
For any fixed $\mu>0$ and almost every $\rho \in\left[\frac12,1\right]$,
there exists $(u_\rho, \lambda_{\rho})\in H^1_\mu(\cG)\times \mathbb{R}^+$ which solves
\begin{equation}\label{pb rho}
\begin{cases}
-u_\rho'' + \lambda_\rho u_\rho = \rho
u_\rho^{p-1}, \quad u_\rho>0 & \text{in $\mathcal{\cG}$}, \\
\sum_{{\rm e} \succ {\rm v}} u_\rho'({\rm v}) = 0 &  \text{for any vertex ${\rm v}$}.
\end{cases}
\end{equation}
Moreover, $E_{\rho}(u_{\rho},\cG)=c_{\rho}$ and
$m(u_\rho) \le 1$.
\end{proposition}

\begin{proof}
For simplicity, we omit the dependence of the functionals $E_\rho(\cdot\,\cG)$ on $\cG$.
We apply Theorem~\ref{Thm:mp_geom} to the family of functionals $E_\rho$, with $E=H^1(\cG)$, $H=L^2(\cG)$, $S_\mu = H^1_\mu(\cG)$, and
$\Gamma$ defined in Lemma~\ref{MP-geometry}. Setting
\[
A(u) = \frac12\int_{\cG} |u'|^2\,\,\mathrm{d}x \quad \text{and} \quad  B(u) = \frac{\rho}{p}\int_{\cG} |u|^p\,\,\mathrm{d}x,
\]
assumption \eqref{hp coer} holds, since we have that
\[
u \in H^1_\mu(\cG), \ \|u\| \to +\infty \quad \implies \quad A(u) \to +\infty.
\]
Let $E'_{\rho}$ and $E''_{\rho}$ denote respectively the free first and second derivatives of $E_{\rho}$. Clearly, $E'_{\rho}$ and $E''_{\rho}$ are both of class $C^1$, and hence locally H\"older continuous, on $H_\mu^1(\cG)$, which implies that Assumption \eqref{Holder} holds.

Thus, taking into account Lemma~\ref{MP-geometry}, by Theorem~\ref{Thm:mp_geom},
for almost every $\rho \in [1/2,1]$, there exist a bounded sequence $\{u_{n, \rho}\} \subset H_\mu^1(\cG)$, that we shall denote simply by $\{u_n\}$ from now on, and a sequence $\{\zeta_n\}\subset \mathbb{R}^+$ with $\zeta_n\to 0^+$, such that
\begin{equation}\label{const crit}
E'_{\rho}(u_n) +  \lambda_n u_n \to 0 \quad \mbox{in the dual of} \quad H^1_{\mu}(\cG ),
\end{equation}
where
\begin{equation}\label{lambda}
\lambda_n:= -\frac{1}{\mu} E'_\rho(u_n) u_n
\end{equation}
and if the inequality
\begin{align}\label{L-Hess crit}
E''_\rho(u_n)[\varphi, \varphi] + \lambda_n ||\varphi||_{L^2(\cG )}^2 < -\zeta_n \|\varphi\|_{H^1(\cG)}^2
\end{align}
holds for any $\varphi \in W_n \setminus \{0\}$ in a subspace $W_n$ of $T_{u_n} S_\mu$, then
${\rm dim}W_n \leq 1$.
In addition,
by diamagnetic inequality (see for example~\cite[Remark 1.4]{BCJS-2022})
since
$u \in H^1_\mu(\cG) \ \Longrightarrow  |u| \in H^1_\mu(\cG)$,
the map $u \mapsto |u|$ is continuous, and
$E_\rho(u) = E_\rho(|u|)$, it is possible to choose $\{u_n\}$ with the property that $u_n \ge 0$ on $\cG$.

The sequence $\{u_n\}$ being bounded,  it follows by (\ref{lambda}) that $\{\lambda_n\}$ is bounded. Then, passing to a subsequence, there exists $\lambda_\rho\in \mathbb{R}$ such that $\lim\limits_{n\to+\infty}\lambda_n=\lambda_{\rho}$.
Taking $\rho_n=\rho$ for all $n\in\N$ in Lemma~\ref{Lem:Abstract}, we conclude the proof.
\end{proof}

We now turn to the case $\rho=1$. Beforehand, let us consider the
\begin{lemma}
For $\mu>0$, for $\rho>0$ and $u\in H^1_\mu(\cG)$, $u>0$ such that
\begin{equation}
\begin{cases}
-u'' + \lambda u = \rho
u^{p-1}, & \text{in $\mathcal{\cG}$}, \\
\sum_{{\rm e} \succ {\rm v}} u'({\rm v}) = 0 &  \text{for any vertex ${\rm v}$},
\end{cases}
\end{equation}
define
$$L(u,\cG):=\frac12\int_\cG (u')^2\,\mathrm{d}x+\frac\rho{p}\int_\cG u^p\,\mathrm{d}x-\frac{\lambda}2\int_\cG u^2\,\mathrm{d}x$$
then
\[
 \left(\frac12-\frac1p\right)L(u,\cG)
 = \left(\frac12+\frac1p\right)E(u,\cG)+
\frac{3\lambda}2\left(\frac1p-\frac16\right)\mu.
 \]
\end{lemma}
\begin{proof}
Note that for any ${\rm e}\in {\mathcal E}$, $u_{\upharpoonright\rm e}\in C^2({\rm e})$ and
multiplying $-u'' + \lambda u_\rho - \rho u^{p-1}$ by $u$ and integrating on $\cG$ provides the Pohozaev identity
\[
 \int_\cG (u')^2 \,\mathrm{d}x+ \lambda \int_\cG u^2 \,\mathrm{d}x- \rho \int_\cG  u^p\,\mathrm{d}x=0
\]
so that
\[
E(u,\cG)= \left(\frac12-\frac1p\right)\int_\cG (u')^2 \,\mathrm{d}x-\frac\lambda{p}\int_\cG u^2 \,\mathrm{d}x
\]
and
\[
L(u,\cG):=\left(\frac12+\frac1p\right)\int_\cG (u')^2\,\mathrm{d}x-\lambda\left(\frac12-\frac1p\right)\int_\cG u^2\,\mathrm{d}x
\]
so that
\begin{align*}
 \left(\frac12+\frac1p\right)E(u,\cG)+
\left(\frac12+\frac1p\right)\frac\lambda{p}\int_\cG u^2 \,\mathrm{d}x
 &=
 \left(\frac12-\frac1p\right)L(u,\cG)\\
 &\quad +
 \left(\frac12-\frac1p\right)
\lambda\left(\frac12-\frac1p\right)\int_\cG u^2\,\mathrm{d}x
\end{align*}
which concludes the proof.
\end{proof}
\begin{corollary}\label{Cor:BoundHamiltonian}
For $\mu>0$, for $\rho>0$ and $u\in H^1_\mu(\cG)$, $u>0$ such that
\begin{equation}
\begin{cases}
-u'' + \lambda u = \rho
u^{p-1}, & \text{in $\mathcal{\cG}$}, \\
\sum_{{\rm e} \succ {\rm v}} u'({\rm v}) = 0 &  \text{for any vertex ${\rm v}$},
\end{cases}
\end{equation}
define for any ${\rm e}\in {\mathcal E}$
$$\ell_{\lambda, \rho}(u,{\rm e}):=\frac12(u')^2+\frac\rho{p}u^p-\frac{\lambda}2u^2,$$
which is constant on ${\rm e}$,
then
\[
 \sum_{{\rm e}\in {\mathcal E},\,\ell_{\lambda, \rho}(u,{\rm e})<0}|{\rm e}||\ell_{\lambda, \rho}(u,{\rm e})|
 \geq -\left(\frac12+\frac1p\right)E(u,\cG)+
\frac{3\lambda}2\left(\frac16-\frac1p\right)\mu.
 \]
\end{corollary}

\begin{proposition}\label{Prop:Main}
For any fixed $\mu>0$
there exists $(u, \lambda)\in H^1_\mu(\cG)\times \mathbb{R}^+$ which solves
\begin{equation}
\begin{cases}
-u'' + \lambda u = u^{p-1}, \quad u>0 & \text{in $\mathcal{\cG}$}, \\
\sum_{{\rm e} \succ {\rm v}} u'({\rm v}) = 0 &  \text{for any vertex ${\rm v}$}.
\end{cases}
\end{equation}
\end{proposition}
\begin{proof}
The thesis will follow as in Proposition~\ref{Prop:NLSrho} considering a sequence
$\rho_n\to 1$ such that the conclusion of
Theorem~\ref{Thm:mp_geom} holds
and the corresponding $(u_n)_{n\in\N}$ and $(\lambda_n)_{n\in\N}$ from Proposition~\ref{Prop:NLSrho}.
In order to apply Lemma~\ref{Lem:Abstract} and conclude, we have to prove that
$(u_n)_{n\in\N}$ and $(\lambda_n)_{n\in\N}$  are bounded in $H^1(\cG)$ and $\R$ respectively.

Note that by Pohozaev identity
\[
 E(u_n,\cG)= \left(\frac12-\frac1p\right)\int_\cG (u_n')^2 \,\mathrm{d}x-\frac{\lambda_n}{p}\int_\cG u_n^2 \,\mathrm{d}x
\]
since $\int_\cG u_n^2 \,\mathrm{d}x=\mu$, $(u_n)_{n\in\N}$
is bounded in $H^1(\cG)$
if and only if $(\lambda_n)_{n\in\N}$
is bounded in $\R$.

So let us assume that, up to a subsequence extraction, $\lambda_n\to +\infty$.
Recall
$$\ell_{\lambda, \rho}(u,{\rm e})=\frac12(u')^2+\frac\rho{p}u^p-\frac{\lambda}2u^2,$$
For each sequence of edges ${\rm e}_n$ such that $\ell_{\lambda_n, \rho_n}(u_n, {\rm e}_n)<0$,
first we prove that $\ell_{\lambda_n, \rho_n}(u_n, {\rm e}_n)=o(\lambda_n^{\frac{p}{p-2}})$. If it were not true then $\ell_{1,1}(\lambda_n^{-\frac1{p-2}}u_n, {\rm e}_n)\not\to 0$. Since $\ell_{1,\rho_n}(\lambda_n^{-\frac1{p-2}}u_n, {\rm e}_n)\in [\frac{\rho_n}p-\frac12,0]$, up to a subsequence extraction, $|u_n|$ stays away from $0$ and hence
$\inf_{{\rm e}_n} u_n \geq \kappa \rho_n^{-\frac1{p-2}}\lambda_n^{\frac1{p-2}}$, for some $\kappa>0$, and
$$\mu^2 \geq \kappa^2  \inf_{{\rm e}\in \mathcal{E}}|{\rm e}|\rho_n^{-\frac2{p-2}}\lambda_n^{\frac2{p-2}}$$
which leads to a contradiction in the limit $n\to\infty$. The same contradiction would hold if there is a
subsequence of $(u_n)_n$ made of constant functions necessarily equal to $\rho^{-\frac1{p-2}}\lambda_n^{\frac1{p-2}}$.

From \eqref{Eq:ToConcludePeriod} and \eqref{Eq:ToConcludeNorm} below (see Appendix), we deduce
\[
\|u_n\|_{L^2({\rm e}_n)}\sim_{n\to\infty}-|{\rm e}_n|\frac{\kappa_p \lambda_n^{\frac2{p-2}}}{2\ln(-\lambda_n^{-\frac{p}{p-2}}\ell_{\lambda_n, \rho_n}(u_n, {\rm e}_n))}
\]
Since for any $\kappa>0$ and $C>0$, if $X>0$ is sufficiently small then $0<-X\ln(X)<\kappa /C$, we deduce,
for any $C>0$ that there exists $N$ such that for $n\geq N$
\[
\|u_n\|_{L^2({\rm e}_n)}\geq  C \lambda_n^{-1}|{\rm e}_n||\ell_{\lambda_n, \rho_n}(u_n, {\rm e}_n)|
\]
Recall Corollary~\ref{Cor:BoundHamiltonian},
\[
 \sum_{{\rm e}\in {\mathcal E},\,\ell_{\lambda_n,\rho_n}(u_n,{\rm e}))<0}|{\rm e}||\ell_{\lambda_n,\rho_n}(u_n,{\rm e}))|
 \geq -\left(\frac12+\frac1p\right)E(u_n,\cG)+
\frac{3\lambda_n}2\left(\frac16-\frac1p\right)\mu
 \]
so that
\begin{align*}
\mu &\geq
C
\lambda_n^{-1}\left(
-\left(\frac12+\frac1p\right)E(u_n,\cG)+
\frac{3\lambda_n}2\left(\frac16-\frac1p\right)\mu\right)
\end{align*}
which leads to
a contradiction
in the limit $n\to\infty$
if $\mu>0$
since $C$ can be chosen such that $C\frac{3}2\left(\frac16-\frac1p\right)>1$.
This contradiction shows $(u_n)_{n\in\N}$ and $(\lambda_n)_{n\in\N}$  are bounded in $H^1(\cG)$ and $\R$ respectively.
Hence Lemma~\ref{Lem:Abstract} applies and we conclude as in Proposition~\ref{Prop:NLSrho}.
\end{proof}
\noindent Theorem~\ref{Thm:Main} follows from Proposition~\ref{Prop:Main} \hfill \qed

\newpage
\appendix

\begin{center}
\large{\bf{ Appendix}}
\end{center}

This appendix is devoted to some useful properties of the ODE problem considered in this work.
This is closely related to the phase space analysis done in~\cite[Section]{Esteban20}. The references listed therein presents other properties of solutions.

\section{Local positive solutions to stationary NLS}

Let us consider on $\R$, real solutions to the ODE:
\begin{equation}\label{Eq:ODE}\tag{ODE}
 -v''-\rho |v|^{p-2}v+m^2 v=0, \quad \rho\in (0,1], \,m \in \R.
\end{equation}
This equation is invariant by translation and symmetry with respect to any point. If $v$ is a solution so is $-v$. The equilibria correspond
to $v\equiv 0$ or constants $v$ such that $\rho|v|^{p-2}=m^2$. \emph{We are interested in $v>0$.}

Note that if $v$ is a solution of~\eqref{Eq:ODE} then $u$ such that $u(x)=\rho^{\frac1{p-2}}m^{-\frac2{p-2}}v(m^{-1}x)$ is a solution of
\[
  -u''-|u|^{p-2}u+u=0.
\]

Let us define the Hamiltonian
\[
 \ell_{m,\rho}(v):=\frac12|v'|^2+\frac{\rho}p|v|^p-\frac{\,\,\,m^2}2|v|^2.
\]
\paragraph{The Hamiltonian well.}
Let
\[
\pi_{m,\rho}(v):=\frac{\rho}p|v|^p-\frac{\,\,\,m^2}2|v|^2.
\]
This is an even function vanishing at $0$ (an equilibria of~\eqref{Eq:ODE}) and $\pm \gamma_+$ where
\[
 \gamma_+:= \left(\frac{p}{2\rho}\right)^{\frac1{p-2}}m^{\frac2{p-2}}.
\]
Its minima are at $\pm \gamma_-$ (two equilibrium of~\eqref{Eq:ODE}) where
$$\gamma_-:=\rho^{-\frac1{p-2}}m^{\frac2{p-2}},$$
with value $\beta:=\left(\frac1p-\frac12\right)\rho^{-\frac2{p-2}}m^{\frac{2p}{p-2}}$,
$0$ is a local maximum (value is $0$) and it tends to $+\infty$ at $\pm \infty$.
The level set of $\pi_{m,\rho}$ at $\mu$ real, $\pi_{m,\rho}^{-1}\{\mu\}$,
\begin{itemize}
 \item is empty if $\mu<\beta$,
 \item contains only $\{\pm\gamma_-\}$ if $\mu=\beta$,
 \item contains only two pairs of opposite reals if $\mu\in (\beta, 0)$,
 \item contains only $\pm \gamma_+$ and $0$ if $\mu=0$,
 \item and contains only a pair of opposite reals if $\mu>0$.
\end{itemize}
Then note that
$\pi_{m,\rho}(x)=\rho^{-\frac{2}{p-2}} m^{\frac{2p}{p-2}}\pi_{1,1}(\rho^{\frac{1}{p-2}}m^{\frac{-2}{p-2}}x)$ and we have
\begin{itemize}
 \item $\pi_{1,1}(y)=0$ if and only if $y=0$ or $y=\left(\frac{p}{2}\right)^{\frac1{p-2}}$;
 \item $\pi_{1,1}'(y)= y^{p-1}-y$ and $\pi_{1,1}'(y)=0$ if and only if $y=0$ or $y=1$.  Note that $\pi_{1,1}(1)=\left(\frac1p-\frac12\right)$;
\end{itemize}
Therefore
$\pi_{1,1}$ is invertible from
$[0,1]$
to
$[\left(\frac1p-\frac12\right), 0]$
with inverse denoted $g$.
Then $g$ is continuous and $g(0)=0$ and
for $z\in \left[\left(\frac1p-\frac12\right), 0\right]$, we have
\begin{equation}\label{Eq:Identity}
 \left(\frac1p g(z)^{p-2}-\frac12\right) g(z)^2=z
\end{equation}
and thus
\[
 g(z)=\sqrt{\frac{-z}{\frac12-\frac1p g(z)^{p-2}}}\sim_{z\to 0-}\sqrt{-2z}.
\]
Hence
for $u\in \left[\left(\frac1p-\frac12\right)\rho^{-\frac2{p-2}}m^{\frac{2p}{p-2}}, 0\right]$
\begin{align*}
\pi_{m,\rho}(x)=u \text{ and } \rho^{\frac1{p-2}}m^{-\frac2{p-2}}x\leq 1
&\Leftrightarrow \pi_{1,1}(\rho^{\frac1{p-2}}m^{-\frac2{p-2}}x)=\rho^{\frac2{p-2}}m^{-\frac{2p}{p-2}}u\\
&\qquad\text{ and } \rho^{\frac1{p-2}}m^{-\frac2{p-2}}x\leq 1\\
&\Leftrightarrow x=\rho^{-\frac1{p-2}}m^{\frac2{p-2}}g(\rho^{\frac2{p-2}}m^{-\frac{2p}{p-2}}u)
\end{align*}
and if $u=o(\rho^{\frac2{p-2}}m^{\frac{2p}{p-2}})$, as $m\to 0$ or $m\to +\infty$, then
$$x\sim \rho^{-\frac1{p-2}}m^{\frac2{p-2}}\sqrt{-2\rho^{\frac2{p-2}}m^{-\frac{2p}{p-2}}u}=m^{-1}\sqrt{-2u}=o(\rho^{\frac1{p-2}}m^{\frac{2}{p-2}}).$$

Rearranging~\eqref{Eq:Identity}, we have
\begin{align*}
\left(\left(\frac1p-\frac12\right)
+\frac1p \frac{g(z)^{p-2}-1-(p-2)(g(z)-1)}{(g(z)-1)^2}g(z)^2\right)(g(z)-1)^2
=&
z-\left(\frac1p-\frac12\right)
\end{align*}
which leads to
\[
g(z)-1\sim_{z\to\left(\frac1p-\frac12\right)-} \sqrt{\frac{2p}{(p-2)(p-4)}\left(z-\left(\frac1p-\frac12\right)\right)}
\]

\paragraph{Properties of solutions.}
Any weak continuous solution is of class $C^2$ and thus there are only strong solutions
and $\ell_{m,\rho}(v)$ is constant on intervals.
Note that $\ell_{m,\rho}(v)\geq \pi_{m,\rho}(v)$. As $\lim_{v\to \pm\infty} \pi_{m,\rho}(v)=+\infty$, all the solutions are bounded. Moreover, we have the following properties.
\begin{enumerate}
 \item If $m=0$ there is no solution with $\ell_{m,\rho}(v)\equiv 0$ except $v\equiv 0$.
 \item Since $|v'|^2=2(\ell_{m,\rho}(v)-\pi_{m,\rho}(v))$, $v'$ is bounded thus any maximal solution is global (defined on $\R$).
 \item Let $v$ be defined on a half line and tend to $k$ at the infinite end. Then $2|v'|^2=\ell_{m,\rho}(v)-\pi_{m,\rho}(v)$ tends to $k'=\ell_{m,\rho}(v)-\pi_{m,\rho}(k)$. Since $v'$ is continuous and
$v$ has a limit, $k'=0$. Then $v''$ tends to $m^2k-\rho|k|^{p-2}k$ and with the same reasoning $m^2k-\rho|k|^{p-2}k=0$ that is $k=0$
or $k=\pm \rho^{-\frac1{p-2}}m^{\frac2{p-2}}$. If $k=\pm \rho^{-\frac1{p-2}} m^{\frac2{p-2}}$, the minimum of $\ell_{m,\rho}$, then $v'\equiv 0$ and $v$ is constant.
  Note that if $v$ is maximal and $v'$ never vanishes, $v$ is monotonic and has limits at both infinities which are different equilibria (this is a heteroclinic orbit and $\ell_{m,\rho}(v)=0$) which is impossible. Indeed, one of this equilibrium is $\gamma_-$ or $-\gamma_-$ and therefore $\ell_{m,\rho}(v)=\pi_{m,\rho}(\pm\gamma_-)=\beta$ and $v$ is constant.
 \item 
 If $v$ is a maximal solution,
 all the critical values of $v$ are in the same level set of $\ell_{m,\rho}$ and are either made of
 two values
 if $\ell_{m,\rho}(v)\neq 0$.
 If $\ell_{m,\rho}(v)<0$, $v$ does not vanish. If $\ell_{m,\rho}(v)>0$, $v$ vanishes somewhere. 
 If $\ell_{m,\rho}(v)=0$, either it $v\equiv 0$ or $|v|$ has a maximum $\left(\frac{p}2\right)^{\frac1{p-2}}\rho^{-\frac1{p-2}}m^{\frac2{p-2}}$ and tends to $0$ at infinity. Indeed, if $0$ is a critical point then the solution is, by uniqueness, identically zero.
\item If $v$ is a maximal solution and has critical points
then $v$ is symmetric with respect to any of its critical point (from uniqueness of the Cauchy problem and the invariance by symmetry and translation of the equation). If $v$ has two distinct critical values then it is periodic and
\begin{itemize}
 \item $\ell_{m,\rho}(v)<0$, $v>0$ non constant, iff its maximum ${v_+}\in (\gamma_-,\gamma_+)$, then there is another critical value ${v_-}\in (0,\gamma_-)$, and ${v_-}$ is the minimum and $v$ is positive (and periodic). Recall $\ell_{m,\rho}(v)<0$ and $v<0$ iff $\ell_{m,\rho}(-v)<0$ and $-v>0$.
 \item $\ell_{m,\rho}(v)>0$ iff its maximum ${v_+}>\gamma_+$, then $-{v_+}$ is its minimum (and $v$ is periodic).
 \item $\ell_{m,\rho}(v)=0$, $v$ non constant iff its maximum is ${v_+}=\gamma_+$, then $v$ is positive and tends to $0$ at infinity. Note that $\ell_{m,\rho}(v)=0$, $v$ constant iff $v\equiv 0$.
 \item $\ell_{m,\rho}(v)=\beta$ iff its maximum ${v_+}=\gamma_-$ and then $v$ is constant.
\end{itemize}
\end{enumerate}

\paragraph{Period with negative Hamiltonian.}
We consider $v$ maximal with $\ell:=\ell_{m,\rho}(v)<0$.
Hence, $v$ has maximal value ${v_+}$ at $x_{v_+}$ and minimal value ${v_-}$ at $x_{{v_-}}$, the smallest minimum larger than $x_{v_+}$ so that
\[
\ell= \pi_{m,\rho}({v_+})
=
\pi_{m,\rho}({v_-})
\]
and if \fbox{$v$ is not constant} then
let us now consider the period $T_{m,\rho}(\ell)$. Recall that the solution is symmetric with respect to its critical points. Then, for $m=\rho=1$,
we have
\begin{align*}
\frac12T_{1,1}(\ell)&=
\int_{x_{v_+}}^{x_{{v_-}}}\,\mathrm{d}x=
 \int_{{v_-}}^{v_+}\frac{\,{\mathrm d}v}{\sqrt{2\ell-\frac2pv^p+v^2}}
 = 2\int_{{v_-}}^{1}\frac{\,{\mathrm d}v}{\sqrt{2\ell-\frac2pv^p+v^2}}\\
 &= \sqrt{2}\int_{{v_-}}^{1}\frac{\,{\mathrm d}v}{\sqrt{\ell-\pi_{1,1}(v)}}=\sqrt{2}\int_{\pi_{1,1}(1)}^{\pi_{1,1}({v_-})}\frac1{\sqrt{\ell-p}}
 \frac{\,{\mathrm d}p}{\pi_{1,1}'(g(p))}\\
 &=\sqrt{2}\int_{\frac1p-\frac12}^{\ell}\frac1{\sqrt{\ell-p}}
 \frac1{g(p)}\frac{\,{\mathrm d}p}{g(p)^{p-2}-1}\\
 &\sim_{\ell\to 0-}-\sqrt{2}\int_{\frac1p-\frac12}^{\ell}\frac1{\sqrt{\ell-p}}
 \frac1{\sqrt{-2p}}\,{\mathrm d}p=\int_1^{\frac{p-2}{2p(-\ell)}}\frac1{\sqrt{p-1}}
 \frac1{\sqrt{p}}\,{\mathrm d}p\\
 &\sim_{\ell\to 0-}-\ln(-\ell).
\end{align*}
We thus infer
\begin{equation}\label{Eq:ToConcludePeriod}
T_{m,\rho}(\ell)=m^{-1}T_1(\rho^{\frac1{p-2}}m^{-\frac{2p}{p-2}}\ell)\sim_{\ell\to 0-}-2m^{-1}\ln(-\rho^{\frac1{p-2}}m^{-\frac{2p}{p-2}}\ell).
\end{equation}
\paragraph{Norm with negative Hamiltonian.}
Now let us consider the square of the $L^2$ norms over a period $N_{m,\rho}(\ell)$. If $m=\rho=1$ then
 \begin{align*}
\frac12 N_{1,1}(\ell)&:=\int_{x_{v_+}}^{x_{{v_-}}}v(x)^2\,\mathrm{d}x=
\int_{{v_-}}^{v_+}\frac{v^2\,{\mathrm d}v}{\sqrt{2\ell-\frac2pv^p+v^2}}=
 2\int_{{v_-}}^{1}\frac{v^2\,{\mathrm d}v}{\sqrt{2\ell-\frac2pv^p+v^2}}\\
 &=\sqrt{2}\int_{\frac1p-\frac12}^{\ell}\frac1{\sqrt{\ell-p}}\frac{g(p)}{g(p)^{p-2}-1}\,{\mathrm d}p
\end{align*}
which is integrable and has a finite limit $\kappa_p$ as $\ell\to 0-$.
We thus infer
\begin{equation}\label{Eq:ToConcludeNorm}
N_{m,\rho}(\ell)=\rho^{-\frac2{p-2}}m^{\frac4{p-2}-1}N_1(\rho^{\frac1{p-2}}m^{-\frac{2p}{p-2}}\ell)\sim_{\ell\to 0-} 2 \kappa_p \rho^{-\frac2{p-2}} m^{\frac4{p-2}-1}.
\end{equation}

\newpage

\bibliographystyle{plain}
\bibliography{biblio}

\end{document}